\documentclass[12pt]{article}

\usepackage{amssymb}
\usepackage{amsthm}
\usepackage{amsmath}
\usepackage{amsfonts}
\usepackage{txfonts}
\usepackage{pxfonts}
\usepackage{hyperref}
\usepackage{color}

\usepackage[T1]{fontenc}
\usepackage[latin2]{inputenc}
\title{Existence of piecewise weak solutions of a discrete Cucker-Smale's flocking model with a singular communication weight}
\author{Jan Peszek\\ 
{\it\small Institute of Applied Mathematics and Mechanics, University of Warsaw,}\\
{\it\small ul. Banacha 2, 02-097 Warsaw, Poland}\\
{\it\small E-mail: j.peszek@mimuw.edu.pl}
}

\date{}

\renewcommand{\it}{\sl}
\renewcommand{\em}{\sl}

plus 6pt minus 12pt
\newcommand{\barint}{
         \rule[.036in]{.12in}{.009in}\kern-.16in
          \displaystyle\int  }
\def\r{{\mathbb{R}}}

\begin{document}

\newtheorem{theo}{\bf Theorem}[section]
\newtheorem{coro}{\bf Corollary}[section]
\newtheorem{lem}{\bf Lemma}[section]
\newtheorem{rem}{\bf Remark}[section]
\newtheorem{defi}{\bf Definition}[section]
\newtheorem{ex}{\bf Example}[section]
\newtheorem{fact}{\bf Fact}[section]
\newtheorem{prop}{\bf Proposition}[section]
\newtheorem{prob}{\bf Problem}[section]

\makeatletter \@addtoreset{equation}{section}
\renewcommand{\theequation}{\thesection.\arabic{equation}}
\makeatother

\newcommand{\ds}{\displaystyle}
\newcommand{\ts}{\textstyle}
\newcommand{\ol}{\overline}
\newcommand{\wt}{\widetilde}
\newcommand{\ck}{{\cal K}}
\newcommand{\ve}{\varepsilon}
\newcommand{\vp}{\varphi}
\newcommand{\pa}{\partial}
\newcommand{\rp}{\mathbb{R}_+}
\newcommand{\hh}{\tilde{h}}
\newcommand{\HH}{\tilde{H}}
\newcommand{\cp}{{\rm cap}^+_M}
\newcommand{\hes}{\nabla^{(2)}}
\newcommand{\nn}{{\cal N}}
\newcommand{\dix}{\nabla_x\cdot}
\newcommand{\dm}{{\rm div}_m}
\newcommand{\di}{{\rm div}}
\newcommand{\pxi}{\partial_{x_i}}
\newcommand{\pmi}{\partial_{m_i}}

\maketitle

\begin{abstract}
We prove existence of global $C^1$ piecewise weak solutions for the discrete Cucker-Smale's flocking model with the communication weight \[\psi(s)=s^{-\alpha}, 0<\alpha<1.\] We also discuss the possibility of finite in time alignment of the velocities of the particles.
\end{abstract}

\section{Introduction}
We consider the Cucker-Smale's (C-S) flocking model describing a collective self-driven motion of self-propelled particles, which for some reason have a tendency to flock, such as flockings of birds, schoolings of fishes or concentration of network activity. The purpose of this paper is to prove existence of piecewise weak solutions for the C-S model with an unbounded and non-Lipschitz communication weight. In the classical, discrete C-S model we assume that we have exactly $N$ particles in some unspecified space and that $(x_i,v_i)$ denotes the position and velocity of $i$-th particle, thus for $i=1,...,N$, we have
\begin{eqnarray}\label{cs1}
\frac{dx_i}{dt}=v_i.
\end{eqnarray}
Moreover we know that velocity of each particle changes in time according to the equation
\begin{eqnarray}\label{cs2}
\frac{dv_i}{dt}=\frac{1}{N}\sum_{j=1}^N\psi(|x_j-x_i|)(v_j-v_i),
\end{eqnarray}
where $\psi:[0,\infty)\to[0,\infty)$ is a given function called  the {\it communication weight} or the {\it communication rate}. The communication weight may be interpreted as perception of the particles. Indeed, we have
\begin{eqnarray*}
\frac{1}{N}\sum_{j=1}^N\psi(|x_j-x_i|)(v_j-v_i)=\frac{1}{N}\sum_{j=1}^N\psi(|x_j-x_i|) \left[\frac{\sum_{j=1}^N\psi(|x_j-x_i|)v_j}{\sum_{j=1}^N\psi(|x_j-x_i|)}-v_i\right],
\end{eqnarray*}
thus $v_i$ evolves towards $\frac{\sum_{j=1}^N\psi(|x_j-x_i|)v_j}{\sum_{j=1}^N\psi(|x_j-x_i|)}$, which is a convex combination of velocities $v_j$ and can be interpreted as the avarage velocity of the particles percieved by $i$-th particle. For instance if $\psi\equiv 1$ then it is exactly the avarage velocity of the particles and if $\psi\equiv 0$ then the particles move independently with constant velocity as if they did not see each other.

\noindent
As the number of particles grows to infinity, i.e., $N\to\infty$, the discrete model is replaced by the following Vlasov-type equation:
\begin{eqnarray}\label{cscont}
\partial_tf+v\cdot\nabla f+{\rm div}_v(F(f)f)=0,\ \ x\in\r^d,\ v\in\r^d,\\
F(f)(x,v,t):=\int_{\r^{2d}}\psi(|y-x|)(w-v)f(y,w,t)dwdy\nonumber,
\end{eqnarray}
where $f=f(x,v,t)$ is a density of particles that at the time $t$ have position $x$ and velocity $v$.
\subsection{Smooth communication weight}
The most well known communication weight is a bounded and smooth one given by
\begin{eqnarray}\label{cucu}
\psi_{cs}(s)=\frac{K}{(1+s^2)^\frac{\beta}{2}}, \ \ \ \beta\geq 0,\ \ \  K>0,
\end{eqnarray}
or more generally $\psi_{cs}$-- bounded and Lipschitz continuous.
The C-S model with weight $\psi_{cs}$ was introduced in 2007 by Cucker and Smale in \cite{cuc1} and was in some sense based on the paper by Viscek \cite{vic} from 1995 in which a model of flocking was introduced, such that each particle adjusted it's velocity with respect to the avarage velocity of it's neighbors. Since then existence, uniqueness, asymptotics and stability for alignment models similar to C-S (both in continuous and discrete cases) were extensively studied, both from physical and biological point of view \cite{deg1} -\cite{deg4}, \cite{le}, \cite{park} - \cite{topaz} and from more theoretical point of view  \cite{aha1, aha2}, \cite{bocan} - \cite{cuc4}, \cite{dua1} - \cite{halaruslem}, \cite{ka, mo}. A nice and thorough study of the C-S flocking model with a bounded communication weight can be found in \cite{hatad}, where the interplay between dicrete and continuous model is studied with measure valued solutions of the Vlasov type equation (\ref{cscont}) or in \cite{haliu}, where the authors present a new, simple aproach to the problem of existence and asymptotics. Recently other direction of studies was introduced - namely to couple a continuous C-S model treated as a Vlasov equation with Navier-Stokes system to model the motion of fluid imbeded particles. This approach is based on the paper by Beduin \cite{org} and can be found in \cite{bae}.
\subsection{Singular communication weight}\label{int}
Our main interest is the C-S model with weight
\begin{eqnarray}\label{psi}
\psi(s)=
\left\{
\begin{array}{ccc}
s^{-\alpha}&{\rm for}& s>0,\\
0&{\rm for}& s=0,
\end{array}
\ \ \ \ \ \ \alpha\in(0,d).\right.
\end{eqnarray}
\noindent
With the loss of Lipschitz continuity of $\psi$ the problem of existence and uniqueness for the discrete C-S model becomes more difficult. To our best knowledge there are no results in this direction, even though the are results on asymptotics in such case, see e.g. \cite{haliu}. When dealing with the C-S model with bounded weight $\psi_{cs}$ one makes use of Lipschitz continuity of $\psi_{cs}$ as well as the structure of the model itself. As an example we will now present a simple aplication of the properties of the structure of our model. Namely we will prove that the avarage velocity of the particles
\begin{eqnarray*}
\bar{v}(t):=\frac{1}{N}\sum_{i=1}^Nv_i(t)
\end{eqnarray*}
 is constant in time. Assuming that $x=(x_1,...,x_N)$ and $v=(v_1,...,v_N)$ is a sufficiently smooth solution of (\ref{cs1}) and (\ref{cs2}), we calculate the derivative of $\bar{v}$ to get
\begin{eqnarray*}
\frac{d}{dt}\sum_{i=1}^Nv_i=\frac{1}{N}\sum_{i,k=1}^N(v_k-v_i)\psi_n(|x_i-x_k|)=\\
\frac{1}{2N}\sum_{i,k=1}^N(v_k-v_i)\psi_n(|x_i-x_k|)+ \frac{1}{2N}\sum_{i,k=1}^N(v_i-v_k)\psi_n(|x_i-x_k|)=0,
\end{eqnarray*}
where the latter summant in the second line is obtained by substituting $i$ and $k$. Clearly each such structure based property of the C-S model will remain true regardless of the communication weight $\psi$ as long as it is a nonnegative function. This is the first piece of information on which we base our hope to obtain some existence for C-S model with singular weight $\psi$. The second piece of information is that Lipschitz continuity and boundedness of $\psi$ fails only at $0$, which means that our main problem will be to prove existence in a neighborhood of each time $t_0$ at which some particles collide. However, heuristically in a neighborhood of each such point we have
\begin{eqnarray*}
x_i(t)-x_j(t)\approx t(v_i(t_0)-v_j(t_0))\approx t(v_i(t)-v_j(t))
\end{eqnarray*}
and since in (\ref{cs2}) the function $t\mapsto\psi(|x_i(t)-x_j(t)|)$ comes always multiplied by $v_i(t)-v_j(t)$, we have
{\small
\begin{eqnarray*}
\left(\Psi(|x_i(t)-x_j(t)|)\right)^{'}=\psi(|x_i(t)-x_j(t)|)\frac{(x_i(t)-x_j(t))\cdot(v_i(t)-v_j(t))}{|x_i(t)-x_j(t)|}\approx \psi(|x_i(t)-x_j(t)|)(v_i(t)-v_j(t))
\end{eqnarray*}
}
with $\Psi(s):=\frac{1}{1-\alpha}s^{1-\alpha}$ being a primitive of $\psi$, which is a H\" older continuous function, thus there is hope for some better regularity of $v$. These two observations were already used in \cite{haliu} to obtain asymptotic flocking for C-S model with weight $\psi$. Occurrence of asymptotic flocking is a further clue that a C-S model with singular weight inherits some nice properties from the model with a smooth weight. Lastly in \cite{haliu} existence for the discrete model served as a mean to obtain existence for the continuous model by defining the solution of the continuous model as a Wasserstein metric's limit of approximative discrete solutions. What is interesting from the point of view of the model with singular weight is that the existence of such limit was ensured by structure only, which means that it should work also in our case. Of course existence of the limit is not enough to prove that the limit actually satisfies any equation at all but it is a first step. In \cite{haliu} it was the Lipschitz continuity of $\psi_{cs}$ that served to prove that this limit was indeed a solution of the continuous C-S model, which indicates that this may be the most difficult part in case of the singular communication weight.

\subsection{Preliminaries and notation}
The definition of our piecewise weak solutions and the proof of their existence and regularity can be found in Chapter 2. In Chapter 3, we restrict ourselves to the case of two particles and present necessary and sufficient conditions on the initial data for the trajectories of the particles to stick together in a finite time.\\
\noindent
Hereinafter $x=(x_1,...,x_N)\in\r^{Nd}$, where $x_i=(x_{i,1},...,x_{i,d})$ denotes the position of the particles, $v=\dot{x}$ is their velocity, where $N$ and $d$ are the number of the particles and the dimension of the space. Approximate solutions $x^n$ and $v^n$ are defined in section \ref{ap}. Moreover by $B_i(t)$ we will denote the set of all indexes $j$, such that the trajectory of $x_j$ does not coincide with the trajectory of $x_i$ as of the time $t$. Assuming that the trajectories, once coinciding cannot separate, we may define it as
\begin{eqnarray}\label{bi}
B_i(t):=\{k=1,...,N: x^n_k(t)\neq x^n_i(t)\ {\rm or}\ v^n_k(t)\neq v^n_i(t)\},
\end{eqnarray}
since any two particles with sufficiently smooth trajectories have the same position and velocity at the time $t$, if and only if they move on the same trajectory. Further, by $W^{k,p}(\Omega)$ we denote the Sobolev space of the functions with up to $k$-th weak derivative belonging to the space $L^p(\Omega)$ and by $C^\infty_c(\Omega)$ we denote the space of smooth and compactly supported functions.\\
\noindent
In the sequel we use the following lemmas, which are well known results from measure theory.
\begin{lem}[Vitali's convergence theorem]\label{vit}
Let $\Omega\subset\r^d$ be bounded and $f_n$ be a sequence of functions converging a.e. to an a.e. finite function $f$. Then
if $f_n$ are uniformly integrable, then $f\in L^1(\Omega)$ and
\begin{eqnarray*}
\lim_{n\to\infty}\int_\Omega|f_n-f|d\mu = 0.
\end{eqnarray*}
\end{lem}

\begin{lem}\label{weakstrong}
Let $\Omega\subset\r^d$ and $f_n,f,g_n,g:\Omega\to\r$ be measurable functions.
If $f_n\to f$ a.e. in $\Omega$, $f_n$ is uniformly bounded in $L^\infty(\Omega)$ and $g_n\rightharpoonup g$ in $L^1(\Omega)$, then
\begin{eqnarray*}
f_ng_n\rightharpoonup fg\ {\rm in\ } L^1(\Omega).
\end{eqnarray*}
\end{lem}
\noindent
Lastly we will say that particles $x_i$ and $x_j$ {\it collide} at the time $t$ if and only if $x_i(t)=x_j(t)$ but $x_i\neq x_j$ in some left-sided neighborhood of $t$ and we will say that they {\it stick} at the time $t$ if and only if they collide and $v_i(t)=v_j(t)$.

\section{Main result}
\subsection{Outline}\label{outline}

In this section we prove existence for the discrete C-S model (\ref{cs1}) and (\ref{cs2}) with a singular communication weight given by (\ref{psi}). Our strategy is based on the observation that the function $t\mapsto \psi(|x_i(t)-x_j(t)|)$ is Lipschitz continuous in a neighborhood of each time $t_0$, such that for all $i,j$, we have $x_i(t_0)\neq x_j(t_0)$, which makes local existence in such points trivial. The idea is that if we can prove that the particles collide in some sense rarely, then the only difficulty will be to establish existence in a neighborhood of each point of collision of some particles. Technicaly we will obtain existence of solutions by approximating them with solutions of C-S model with bounded weights.  

\subsection{Approximate solutions}\label{ap}
In this section we define the approximate solutions and present some of their most important properties. For each $n$ let
\begin{eqnarray*}
\psi_n(s)=\left\{
\begin{array}{ccc}
\psi(s) & {\rm if}& s\geq (n-1)^{-\frac{1}{\alpha}},\\
smooth\ and\ monotone & {\rm if}& n^{-\frac{1}{\alpha}}\leq s\leq (n-1)^{-\frac{1}{\alpha}},\\
n & {\rm if}& s\leq n^{-\frac{1}{\alpha}}
\end{array}
\right.
\end{eqnarray*}
for all $s\in [0,\infty)$ with $\psi$ given by (\ref{psi}). For all $n$, functions $\psi_n$ are smooth and bounded, thus C-S systems associated with these weights have unique solutions. This can be expressed by the following proposition.
\begin{prop}\label{psin}
For each positive integer $n$ and for arbitrary initial data, the system
\begin{eqnarray}\label{n}
\left\{
\begin{array}{lll}
\dot{x}_i^n=&v_i^n,\\
\dot{v}_i^n=&\frac{1}{N}\sum_{k=1}^N(v_k^n-v_i^n)\psi_n(|x_i^n-x_k^n|)
\end{array}
\right.
\end{eqnarray}
has a unique global classical solution $x^n$ belonging to the class $(C^2([0,T]))^{Nd}$. Moreover, this solution is stable with respect to the initial data i.e. for all $x^{n,1}(0),x^{n,2}(0),v^{n,1}(0),v^{n,2}(0)$, we have
\begin{eqnarray*}
\|x^{n,1}-x^{n,2}\|_{W^{1,\infty}([0,T])}< L(n)\left(|x^{n,1}(0)-x^{n,2}(0)|+|v^{n,1}(0)-v^{n,2}(0)|\right),
\end{eqnarray*}
where $L(n)$ is a constant depending on $n$ (and also on $T$, $d$, $\alpha$ and $N$).
\end{prop}
\noindent
The proof of this proposition is standard and we omit it.
\begin{rem}\rm\label{stab}
Stability with respect to the initial data is uniform in a neighborhood of each time in which no collision occurs. Indeed, if in some interval $[0,t]$ we have
\begin{eqnarray*}
\inf_{i,j}|x_i^n-x_j^n|>\delta
\end{eqnarray*}
for all $n$, then C-S systems associated with $\psi_n$ are exactly the same for all $n>\delta^{-\alpha}+1$. It follows from the fact that for $n>\delta^{-\alpha}+1$ and $s\geq \delta$ all functions $\psi_n$ coincide. It implies that if we consider functions $x^n$ restricted to $[0,t]$, then Lipschitz continuity with respect to the initial data mentioned in Proposition \ref{psin} holds with an $n$ independent constant $L(\lfloor\delta^{-\alpha}\rfloor+2)$ which corresponds to the C-S system associated with the smallest weight for which $\psi_n$ coincide for $s\geq \delta$.
\end{rem}
\noindent
Next, let us state some properties of solutions of the C-S model associated with $\psi_n$.
\begin{prop}\label{properties}
Let $x^n$ be a solution of the C-S model associated with weight $\psi_n$. Then $x^n$ has the following properties:
\begin{enumerate}
\item It belongs to the class $C^\infty$ in a neighborhood of every such point $t$, that $|x^n_i(t)-x^n_j(t)|>0$ for all $i,j=1,...,N$.
\item The avarage velocity of the particles is constant:
\begin{eqnarray*}
\frac{1}{N}\sum_{i=1}^Nv_i(t)=const.
\end{eqnarray*}
\item Velocity $v^n$ is bounded: there exists a constant $M(n)$ such that for all $i=1,...,N$, we have
\begin{eqnarray*}
\|v^n_i\|_{L_\infty([0,T])}\leq M(n).
\end{eqnarray*}
\item If the initial data $x^n(0), v^n(0)$ are uniformly bounded, then also $v^n$ is uniformly bounded: there exists a constant $M$ such that for all $i=1,...,N$ and all $n=1,2,...$, we have
\begin{eqnarray*}
\|v^n_i\|_{L_\infty([0,T])}\leq M.
\end{eqnarray*}
\item Acceleration $\dot{v}_i^n$ is bounded by $2M(n)n$.
\item If at some point $t$ we have $x_i^n(t)=x_j^n(t)$ and $v_i^n(t)=v^n_j(t)$ for any $i,j=1,...,N$, then $x^n_i\equiv x^n_j$ on $[t,T]$.
\item If at some point $t$ we have $v^n_i(t)=v^n_j(t)$ for all $i,j=1,...,N$, then $v^n$ is constant on $[t,T]$. 
\end{enumerate}
\end{prop}
\begin{proof}
\begin{enumerate}
\item  Since $x^n$ is continuous, if at some point $t$ all the particles have different positions i.e. $|x^n_i(t)-x^n_j(t)|>0$ for all $i,j=1,...,N$ then it is also true in some neighborhood of $t$. Moreover in this neighborhood of $t$ the right-hand side of $(\ref{n})_2$ is differentiable, which by iteration implies that $x^n$ is smooth at $t$.
\item This part was already done in section \ref{int}.
\item Let $r_n(t):=\sum_{i,j=1}^N(v_i^n(t)-v_j^n(t))^2$. By $(\ref{n})_2$, we have
\begin{eqnarray*}
r_n^{'}=2\sum_{i,j=1}^N(v_i^n-v_j^n)\left(\frac{1}{N}\sum_{k=1}^N(v_k^n-v_i^n)\psi_n(|x_i^n-x_k^n|)- \frac{1}{N}\sum_{k=1}^N(v_k^n-v_j^n)\psi_n(|x_j^n-x_k^n|)\right)=\\
= \frac{2}{N}\sum_{i,j,k=1}^N(v_i^n-v_j^n)(v_k^n-v_i^n)\psi_n(|x_i^n-x_k^n|)- \frac{2}{N}\sum_{i,j,k=1}^N(v_i^n-v_j^n)(v_k^n-v_j^n)\psi_n(|x_j^n-x_k^n|).
\end{eqnarray*}
Again, we substitute $i$ and $k$ in the first summant and $j$ and $k$ in the second summant to obtain
\begin{eqnarray*}
r_n^{'}=\frac{1}{N}\sum_{i,j,k=1}^N(v_i^n-v_j^n)(v_k^n-v_i^n)\psi_n(|x_i^n-x_k^n|) +\frac{1}{N}\sum_{i,j,k=1}^N(v_k^n-v_j^n)(v_i^n-v_k^n)\psi_n(|x_i^n-x_k^n|)\\
-\frac{1}{N}\sum_{i,j,k=1}^N(v_i^n-v_j^n)(v_k^n-v_j^n)\psi_n(|x_j^n-x_k^n|) -\frac{1}{N}\sum_{i,j,k=1}^N(v_i^n-v_k^n)(v_j^n-v_k^n)\psi_n(|x_j^n-x_k^n|)=\\
-\frac{1}{N}\sum_{i,j,k=1}^N(v_i^n-v_k^n)^2\psi_n(|x_i^n-x_k^n|) -\frac{1}{N}\sum_{i,j,k=1}^N(v_j^n-v_k^n)^2\psi_n(|x_j^n-x_k^n|)=\\
=-2\sum_{i,j=1}^N(v_i^n-v_j^n)^2\psi_n(|x_i^n-x_j^n|)\leq 0.
\end{eqnarray*}
Thus for each $n$, function $r_n$ is nonincreasing with it's maximum at $0$ i.e. $r_n(t)\leq r_n(0)$. Now let $\bar{v}^n$ be the avarage velocity, which as we know from property 2 is a constant. We have
{\small
\begin{eqnarray*}
\sum_{i=1}^N(\bar{v}^n-v^n_i)^2=\sum_{i=1}^N\left(\frac{1}{N}\sum_{j=1}^Nv_j^n-v_i^n\right)^2=\frac{1}{N^2} \sum_{i=1}^N\left(\sum_{j=1}^N(v_j^n-v_i^n)\right)^2
\leq \frac{1}{N}\sum_{i,j=1}^N(v_j^n-v_i^n)^2=\frac{1}{N}r_n(0).
\end{eqnarray*}
}
Lastly we have
\begin{eqnarray*}
|v^n_i|\leq |v^n_i-\bar{v}^n|+|\bar{v}^n|
\leq C(N)\sqrt{r_n(0)}+|\bar{v}^n|\leq C(N)\sqrt{r_n(0)}=:M(n),
\end{eqnarray*}
where $C(N)$ is a generic constant depending on $N$.
\item We simply note that if initial velocity is uniformly bounded, then $M(n)\leq M$ for some $M$ independent of $n$.
\end{enumerate}
\noindent
Point 5 follows immediately from property 3 and equation $(\ref{n})_2$, while points 6 and 7 are obvious consequences of uniqueness of the solutions.
\end{proof}
\begin{rem}\rm\label{setb}
Property 6 from the above proposition implies that the acceleration equation $(\ref{n})_2$ can be replaced by:
\begin{eqnarray}\label{cs3}
\dot{v}^n_i=\frac{1}{N}\sum_{k\in B_i(t)}(v^n_k-v^n_i)\psi_n(|x_k^n-x^n_i|),
\end{eqnarray}
where $B_i(t)$ is defined by (\ref{bi}),
with
\begin{eqnarray*}
\dot{v}^n_i=0
\end{eqnarray*}
should set $B_i(t)$ be empty. This technical observation will be useful later on.
\end{rem}
\noindent
Hereinafter we will use $M(n)$ and $M$ in the same roles as in Proposition \ref{properties}.
We end this section with an important lemma that is in fact our way to deal with existence in a right sided neighborhood of a point of collision.

\begin{lem}\label{equi}
Let $x^n$ be a solution of C-S system on the time interval $[0,T]$ with weight $\psi_n$ and initial data $x(0),v(0)$ -- independent of $n$. Then there exists an interval $[0,t]$, such that all velocities $v^n$ are uniformly H\" older continuous on $[0,t]$.
\end{lem}
\noindent
To prove this lemma we need yet another, technical lemma.
\begin{lem}\label{equitech}
If $x_i(0)=x_j(0)$, then for all $n$, there exists an interval $(0,t_n]$, such that
\begin{eqnarray*}
|v_i^n(s)-v_j^n(s)|\leq 4\frac{|(v_i^n(s)-v^n_j(s))\cdot(x^n_i(s)-x^n_j(s))|}{|x^n_i(s)-x^n_j(s)|}
\end{eqnarray*}
for $s\in(0,t_n]$.
\end{lem}
\begin{proof}
By property 5 from Proposition \ref{properties}, we have
\begin{eqnarray}\label{rn}
v^n_i(s)-v^n_j(s)=v^n_i(0)-v^n_j(0)+r_n(s),
\end{eqnarray}
where $|r_n(s)|\leq 2|s|Mn$. Moreover as $x^n_i-x^n_j$ is a $C^2$ function, by Taylor's formula
\begin{eqnarray}\label{on}
x^n_i(s)-x^n_j(s)=s\cdot\left(v^n_i(0)-v^n_j(0)\right)+o_n(s)=s\left(v^n_i(s)-v^n_j(s)-r_n(s)\right)+o_n(s),
\end{eqnarray}
where 
\begin{eqnarray*}
o_n(s):=\int_0^s (\dot{v}^n_i-\dot{v}^n_j)(s-\theta)d\theta,\ \ \ \ \ \ |o_n(s)|\leq 2|s|^2Mn.
\end{eqnarray*}
Thus
{\small
\begin{eqnarray}
|(v_i^n(s)-v^n_j(s))(x^n_i(s)-x^n_j(s))|=|s(v_i^n(s)-v^n_j(s))^2-s(v_i^n(s)-v^n_j(s))r_n(s)+ (v_i^n(s)-v^n_j(s))o_n(s)|\geq\nonumber\\
\geq s(v_i^n(s)-v^n_j(s))^2-s|(v_i^n(s)-v^n_j(s))r_n(s)|-|(v_i^n(s)-v^n_j(s))o_n(s)|\geq \frac{s}{2}(v_i^n(s)-v^n_j(s))^2\label{licznik}
\end{eqnarray}
}
assuming that $s\in(0,t_n]$, where $t_n$ is the supremum of all times $s_n$, such that for all $s\in(t,s_n]$, we have
\begin{eqnarray}\label{war1}
|r_n(s)|\leq \frac{1}{4}|v_i^n(s)-v^n_j(s)|,\ \ \ \ \ \ |o_n(s)|\leq \frac{s}{4}|v_i^n(s)-v^n_j(s)|.
\end{eqnarray}
To check that $t_n>0$, we notice that for
\begin{eqnarray}\label{tn}
s_n:=\frac{|v^n_i(0)-v^n_j(0)|}{10M(n)n}
\end{eqnarray}
and $s\in [0,s_n]$, we have
\begin{eqnarray}\label{war2}
|r_n(s)|\leq \frac{1}{5}|v^n_i(0)-v^n_j(0)|,\ \ \ \ \ \ |o_n(s)|\leq \frac{s}{5}|v^n_i(0)-v^n_j(0)|,
\end{eqnarray}
which together with (\ref{rn}) implies that
\begin{eqnarray*}
\frac{4}{5}|v^n_i(0)-v^n_j(0)|\leq |v^n_i(s)-v^n_j(s)|,
\end{eqnarray*}
condition (\ref{war1}) is satisfied. Therefore by taking $s_n$ given by (\ref{tn}) we get (\ref{licznik}). Now by (\ref{rn}) and (\ref{on}) on $(0,s_n]$ we also have
\begin{eqnarray}\label{war3}
|x^n_i(s)-x^n_j(s)|\leq 2s|v^n_i(s)-v^n_j(s)|, 
\end{eqnarray}
which together with (\ref{licznik}) proves that there exists $s_n>0$ such that on $(0,s_n]$ the assertion holds. Now we define $t_n$ as the supremum of all such times $s_n$. This finishes the proof.
\end{proof}
\noindent
Next we can proceed with the proof of Lemma \ref{equi}.
\begin{proof}[Proof of Lemma \ref{equi}]
The proof will follow by 2 steps. In step 1 we prove that for each $n$ there exists an interval $[0, t_n]$ on which $v^n$ is H\"older continuous with a constant idependent of $n$, while in step 2 we establish a lower bound on $t_n$ that is independent of $n$.\\
{\sc Step 1.} 
It suffices to show the assertion separately for all particles, thus let us fix $i=1,...,N$ and consider $x_i$. By Remark \ref{setb} for all $s$, we have
\begin{eqnarray*}
|v^n_i(s)-v^n_i(0)|=\left|\int_0^s\dot{v}^n_i(\theta)d\theta\right|\leq 
\frac{1}{N}\sum_{k\in B_i(0)}\int_0^s|v^n_k-v^n_i|\psi_n(|x^n_k-x^n_i|)d\theta = \\
\frac{1}{N}\sum_{k\in B^0_i}\int_0^s|v^n_k-v^n_i|\psi_n(|x^n_k-x^n_i|)d\theta + \frac{1}{N}\sum_{k\in B^+_i}\int_0^s|v^n_k-v^n_i|\psi_n(|x^n_k-x^n_i|)d\theta =: I+II,
\end{eqnarray*}
where
\begin{eqnarray*}
B^0_i:=\{j\in B_i(0): |x_j(0)-x_i(0)|=0\},\ \ \ B^+_i:=\{j\in B_i(0): |x_j(0)-x_i(0)|>0\}
\end{eqnarray*}
and $B_i(0)$ is the defined by (\ref{bi}) set of all particles that have different trajectories than $x_i$. Thus $B^0_i$ consists of all particles that start from the same position as $x_i$ but with different velocities, while $B^+_i$ consists of all particles that start from a different position than $x_i$. We may assume that $B_i(0)$ is not empty as otherwise all $x^n$ are constantly equal $x^n(0)$ and the assertion holds. Thus at least one of sets $B^0_i$ or $B^+_i$ is nonempty. Now we estimate $I$ and $II$ separately starting with $I$. For $j\in B^0_i$, we have $|v^n_j(0)-v^n_i(0)|>0$ and by its continuity there exists $t_n$ such that $|x^n_j(s)-x^n_i(s)|>0$ and consequently $\psi_n(|x^n_j(s)-x^n_i(s)|)\leq \psi(|x^n_j(s)-x^n_i(s)|)$ in $(0,t_n]$. Together with Lemma \ref{equitech} it implies that
\begin{eqnarray*}
I\leq \frac{4}{N}\sum_{j\in B^0_i}\int_0^s\frac{|(v^n_j-v^n_i)\cdot(x^n_j-x^n_i)|}{|x^n_j-x^n_i|}\psi(|x^n_j-x^n_i|)d\theta.
\end{eqnarray*}
We claim that, since $\Psi(|x^n_i(0)-x^n_j(0)|)=0$ for all $j\in B^0_i$, then
\begin{eqnarray*}
\int_0^s\frac{|(v^n_j-v^n_i)\cdot(x^n_j-x^n_i)|}{|x^n_j-x^n_i|}\psi(|x^n_j-x^n_i|)d\theta= \Psi(|x^n_j(s)-x^n_i(s)|),
\end{eqnarray*}
where $\Psi(s)=\frac{1}{1-\alpha}s^{1-\alpha}$ is a primitive of $\psi$. Indeed, we have
\begin{eqnarray*}
\Psi(|x^n_j(s)-x^n_i(s)|)=\int_0^s\Psi(|x^n_j-x^n_i|)^{'}d\theta\leq \int_0^s\psi(|x^n_j-x^n_i|)\frac{|(x^n_j-x^n_i)(v^n_j-v^n_i)|}{|x^n_j-x^n_i|}d\theta
\end{eqnarray*}
and since $\psi\geq 0$ we can substitute the above inequality with an equality provided that on 
$(0,t_n]$ the function $\xi(s):=(x^n_j(s)-x^n_i(s))(v^n_j(s)-v^n_i(s))$ has a constant sign. To prove that $\xi$ has a constant sign it suffices to show that $|\xi|>0$ in $(0,t_n]$, which is an immediate consequence of Lemma \ref{equitech}. Thus we proved that
{\small
\begin{eqnarray*}
I\leq\frac{4}{N}\sum_{j\in B^0_i}\Psi(|x^n_j(s)-x^n_i(s)|)= \frac{4}{N(1-\alpha)}\sum_{j\in B^0_i}\left|(x^n_j(s)-x^n_i(s))\right|^{1-\alpha}\leq \frac{4M^{1-\alpha}}{N(1-\alpha)}\sum_{j\in B^0_i}|s|^{1-\alpha}\leq
\frac{4M^{1-\alpha}}{1-\alpha}|s|^{1-\alpha},
\end{eqnarray*} 
}
where we use inequality $|x^n_j(s)-x^n_i(s)|\leq M|s|$ that follows from property 4 from Proposition \ref{properties}. To estimate $II$ we first notice that since for all $j\in B^+_i$, we have $|x^n_j(0)-x^n_i(0)|>0$ then there exists $\delta>0$ such that $|x^n_j(0)-x^n_i(0)|>\delta$ for all $j\in B^+_i$. Then, by property 4 from Proposition \ref{properties} there exists an $n$ independent interval $[0,t_0]$ on which $|x^n_j-x^n_i|>\delta$ for all $j\in B^+_i$. On this interval
\begin{eqnarray*}
\psi_n(|x^n_j(s)-x^n_i(s)|)\leq \delta^{-\alpha}.
\end{eqnarray*}
Therefore
\begin{eqnarray*}
II\leq \frac{1}{N}\sum_{j\in B^+_i} 2|s| M\delta^{-\alpha}\leq 2 t_0^\alpha M\delta^{-\alpha}|s|^{1-\alpha}
\end{eqnarray*}
and adding our estimations of $I$ and $II$ we get
\begin{eqnarray*}
|v_i^n(s)-v_i^n(0)|\leq L |s|^{1-\alpha}
\end{eqnarray*}
with $L=\frac{4M^{1-\alpha}}{1-\alpha}+2t_0^\alpha M\delta^{-\alpha}$ on interval $[0,t_n]\cap[0,t_0]$. For simplicity let us denote $\min\{t_n,t_0\}$ again by $t_n$. This finishes step 1.\\
{\sc Step 2.} In step 1 we proved that for each $n$ there exists an interval $[0,t_n]$ in which $v^n_i$ is H\"older continuous with a constant independent of $n$. Now we prove that there exists $t>0$, such that for all $n$, we have $t\leq t_n$ and thus in $[0,t]$ all functions $v^n_i$ are uniformly H\"older continuous. There are exactly 3 instances, when we bound $t_n$ from the above:
\begin{enumerate}
\item In the proof of Lemma \ref{equitech}.
\item While ensuring that for all $k\in B^0_i$ we have $|v^n_k-v^n_i|>0$ in $[0,t_n]$.
\item While ensuring that for all $k\in B^0_i$ the function $\xi$ is positive in $(0,t_n]$. 
\end{enumerate}  
If each of these bounds from above can be bounded from below by a constant independent of $n$, then so can be $t_n$ for all $n$.
\begin{enumerate}
\item In Lemma \ref{equitech}, $t_n$ was the supremum of all times $s_n$, such that for all $s\in(0,t_n]$ conditions (\ref{war1}) and (\ref{war3}) are satisfied. However from step 1 we may estimate $t_n$ better than we could in the proof of Lemma \ref{equitech}. We have
\begin{eqnarray*}
|r_n(s)|\leq 2L|s|^{1-\alpha}\ \ \ {\rm and}\ \ \ |o_n(s)|\leq 2L|s|^{2-\alpha},
\end{eqnarray*}
thus by taking
\begin{eqnarray}\label{tn1}
t_0:=\left(\frac{1}{10L}|v^n_k(0)-v^n_i(0)|\right)^{\frac{1}{1-\alpha}}
\end{eqnarray} 
we ensure that (\ref{war2}) and consequently (\ref{war1}) is satisfied.
With the same $t_0$ we obtain also condition (\ref{war3}).
\item For $k\in B^0_i$ we have $|v^n_k(0)-v^n_i(0)|> 0$, thus
\begin{eqnarray*}
|v^n_k(s)-v^n_i(s)|\geq |v^n_k(0)-v^n_i(0)|-2L|s|^{1-\alpha},
\end{eqnarray*}
which is positive for $s\leq t_0$ with $t_0$ defined by (\ref{tn1}).
\item To prove that $\xi$ has a constant sign in $[0,t_n]$ we applied Lemma \ref{equitech} concluding that $|\xi(s)|$ is positive, provided that $s$ belongs to the interval on which the thesis of Lemma \ref{equitech} holds and we proved above that this interval includes $(0,t_0]$.
\end{enumerate}
Therefore all bounds from points 1,2 and 3 are satisfied for $t_0$ defined by (\ref{tn1}) and it is clearly $n$-independent. Thus we proved that there exists an interval $[0,t]$ with $t\geq t_0$ in which all functions $v_i^n$ are uniformly H\" older continuous.
\end{proof}

\subsection{Definition of the solution}
Before we define the solution let us recall property 6 from Proposition \ref{properties}, which basically states that the trajectories of the particles cannot separate if they stick together at some point. This is an obvious consequence of the uniqueness for the approximate solutions. However, since $\psi$ is singular at $0$ it may happen that the solutions of the (C-S) model with $\psi$ are not unique and that the trajectories may split as in the case of the well known example $\dot{y}=cx^\frac{1}{3}$. In fact a loss of uniqueness may happen at each time $t$, such that there exist $i$ and $j$, such that $x_i(t)=x_j(t)$. It is problematic because such times $t$ include not only each time of a collision but also each time at which some particles are stuck together. Thus if for example two particles $x_i$ and $x_j$ start with the same position and velocity, then we may lose uniqueness at an arbitrary time $t>0$. Therefore we will enforce that the once stuck trajectories cannot separate. We will do this by replacing equation (\ref{cs2}) with (\ref{cs3}), which does not distinguish trajectories that once stuck together. Hereinafter we consider (C-S) model defined by (\ref{cs1}) and (\ref{cs3}). For this model we still do not have uniqueness but the times at which we lose it are restricted only to the times of collisions, which as we will prove occur in some sense rarely.\\
Thus our problem and it's solution is defined as follows.
\begin{defi}\label{sol}
Let $\{T_n\}_{n\in{\mathbb N}}\cup \{0\}$ be the set of all times of collision of some particles and for each $n$ let $0<T_n\leq T_{n+1}$. For $n\geq -1$, on each interval $[T_n,T_{n+1}]$ (we assume that $T_{-1}=0$) we consider the problem
\begin{eqnarray}\label{cs}
\left\{
\begin{array}{lll}
\frac{dx_i}{dt}=v_i,\\
\frac{dv_i}{dt}=\frac{1}{N}\sum_{k\in B_i(T_n)}(v^n_k-v^n_i)\psi_n(|x_k^n-x^n_i|),
\end{array}
\right.
\end{eqnarray}
for $t\in[T_n,T_{n+1}]$, with initial data $x(T_n), v(T_n)$.\\
\noindent
We say that $x$ solves (\ref{cs}) on the time interval $[0,T]$, with weight given by (\ref{psi}) and arbitrary initial data $x(0), v(0)$ if and only if for all $T_n$ and all $t\in(T_n,T_{n+1})$ the function $x\in (C^1([0,T]))^{Nd}$ is a weak in $(W^{2,1}([T_{n},t]))^{Nd}$ solution of (\ref{cs}), the initial data are correct (i.e. $x(0)=x(T_{-1})$ and $v(0)=v(T_{-1})$) and for some $n$, we have $T\leq T_n$.
\end{defi}
\noindent
This definition may not be clear at the first glance. It is somewhat weaker than a weak solution but stronger than an a.e. solution. Such definition is based on the idea described in section \ref{outline}: the solution exists in a weak sense between two collision times $T_{n-1}$ and $T_n$. However as it approaches $T_n$, the second derivative of $x$ may blow up. Despite this $v$ is still continuous in a left sided neighborhood of $T_n$ and has a limit at $T_n^-$. Therefore we may continuously define it at $T_n$ ensuring existence of unique initial data for local weak existence in $[T_n,T_{n+1})$.

\subsection{Existence up to the time of collision}\label{exi}
Before we begin let us state the following simple remark.
\begin{rem}\rm\label{sususus}
Property 4 from Proposition \ref{properties} implies equicontinuity of $x^n$, thus by Arzela-Ascoli theorem there exists a $(C([0,T]))^{Nd}$ convergent subsequence $x^{n_k}$. From this point we pick one of such convergent subsequences and aim to prove that it has a $(C^1([0,T]))^{Nd}$ convergent subsequence. For simplicity of notation we will assume that $x^n=x^{n_k}$. 
\end{rem}
\noindent
In this section we will prove that the approximate solutions converge in every interval $[0,t]\subset[0,T_0)$, where $T_0$ is the time of the first collision of the particles. Additionaly we will prove that their limit is a weak solution in $(W^{2,1}([0,t]))^{Nd}$. Let us begin with defining $T_0$ by means of the approximate solutions:
\begin{eqnarray*}
T_0:=\inf\{t>0: \min_{\stackrel{i=1,...,N}{j\in B_i(0)}}\lim_{n\to\infty}|x_i^n-x_j^n|=0\}.
\end{eqnarray*}
Note that the limit in the above definition exists, since we are restricted to a $(C([0,T]))^{Nd}$ convergent subsequence.
\begin{rem}\rm\label{t0}
Clearly if $t<T_0$ then there exists $\delta>0$, such that
\begin{eqnarray*}
\min_{\stackrel{i=1,...,N}{j\in B_i(0)}}\lim_{n\to\infty}|x_i^n-x_j^n|>\delta,
\end{eqnarray*}
which further implies that for all $i,j$, there exists $n_0$ such that for all $n>n_0$, we have $|x_i^n-x_j^n|>\delta$. On the other hand
\begin{eqnarray*}
\lim_{n\to\infty}|x_i^{n}(T_0)-x_j^{n}(T_0)|=0
\end{eqnarray*}
and assuming that  $x$ is a $(C([0,T_0]))^{Nd}$ limit of $x^{n}$, we have $x_i(T_0)=x_j(T_0)$, which means that $T_0$ is a point of collision for $x$.
\end{rem}

\begin{prop}\label{odc}
For $n=1,2,...$ let $x^n$ be a solution to C-S system on interval $[0,T]$ with weight $\psi_n$ and an independent of $n$ initial data $x(0)$ and $v(0)$. There exists an interval $[0,T_0)$ such that for any $[0,t]\in[0,T_0)$ solutions $x^n$ have a subsequence that converges in $(C^{1}([0,t]))^{Nd}$.
\end{prop}
\begin{proof}
If for all $i,j=1,...,N$ initial velocity $v_i(0)=v_j(0)$ then by property 7 from Proposition \ref{properties}, we have $v^n\equiv v(0)$ for all $n$ and the assertion holds with $T_0=T$. From this point we assume that there exist $i,j=1,...,N$ such that $v_i(0)\neq v_j(0)$.
Recall $B_i(0)$ defined by (\ref{bi}) -- the set of all indexes which are directly involved in the evolution of $v_i^n$ -- from this point we will only consider $j\in B_i(0)$.
There are two possibilities:
\begin{description}
\item[(A)] For all $i$ and $j\in B_i(0)$ we have $x_j(0)\neq x_i(0)$.
\item[(B)] There exists $i$ and $j\in B_i(0)$ such that $x_j(0)=x_i(0)$.
\end{description}
{\bf (A)} In this case there exists $\delta$ such that for all $i=1,..,N$ and all $j\in B_i(0)$ we have $|x_i(0)-x_j(0)|>\delta$ and by Remark \ref{t0} for all $t<T_0$ there exist $\delta_t\in(0,\delta]$, such that for all $i,j$ we have $|x_i-x_j|>\delta_t$ on $[0,t]$, which implies that $\psi(|x_i-x_j|)\leq \delta^{-\alpha}$ and all velocities $v^n$ are uniformly Lipschitz continuous on $[0,t]$. Therefore by Arzela-Ascoli theorem there exists a $(C^1([0,t]))^{Nd}$ convergent subsequence of $x^n$.\\
{\bf (B)}
In the second case, there exist $i$ and $j$, such that $x_i(0)=x_j(0)$ and $v_i(0)\neq v_j(0)$ and we may not proceed as in case {\bf (A)} as there is no such $\delta>0$, that $|x_i(0)-x_j(0)|>\delta$. However for this situation we have prepared Lemma \ref{equi}, which implies uniform H\" older continuity of $v^n$ in some neighborhood of $0$. Therefore for sufficiently small $t_0$ and $j\in B_i(0)$, such that $x_j(0)=x_i(0)$, we have
\begin{eqnarray*}
|x^n_i(s)-x^n_j(s)|\geq s\left(|v_i(0)-v_j(0)|-2Ls^{1-\alpha}\right)\geq s\frac{1}{2}|v_i(0)-v_j(0)|=:\delta_s>0
\end{eqnarray*}
for $s\in [0,t_0]$. On the other hand, for all $j\in B_i(0)$ such that $x_j(0)\neq x_i(0)$, from property 4 of Proposition \ref{properties}, we have
\begin{eqnarray*}
|x^n_i(s)-x^n_j(s)|\geq |x_i(0)-x_j(0)|-2Ms\geq \delta_s
\end{eqnarray*}
for all $n=1,2,...$ and all $s\in[0,t_1]$ with $0<t_1<1$ possibly smaller than $t_0$. Thus in $t_1$ we end up in a situation from case {\bf (A)} with
\begin{eqnarray*}
|x^n_i(t_1)-x^n_j(t_1)|\geq\delta_{t_1}
\end{eqnarray*}
and all velocities $v^n$ are uniformly H\" older continuous on $[0,t_1]$ and uniformly Lipschitz continuous on $[t_1,t]$ for all $t<T_0$. Again by Arzela-Ascoli theorem, there exists a $(C^1([0,t]))^{Nd}$ convergent subsequence of $x^n$. 
\end{proof}
\begin{rem}\rm\label{subseq}
As in Remark \ref{sususus}, even though $x$ from Proposition \ref{odc} is a limit of some subsequence of $x^n$, we will assume that it is in fact a limit of the whole sequence $x^n$ (by restricting the approximate solutions to only those, which approximate $x$). Such assumption will pose no threat to our reasonings as long as they will not involve uniqueness of $x$.
\end{rem}
\noindent

\begin{coro}\label{coroodc}
Let $x$ be as in Remark \ref{subseq}. Then $x$ is a local classical solution to C-S system in the interval $(0,T_0)$. Moreover
\begin{enumerate}
\item For all $i,j=1,...,N$, we have $|x_j-x_i|>0$ in $(0,T_0)$. 
\item The function $x$ is smooth in $(0,T_0)$.
\end{enumerate}
\end{coro}
\begin{proof}
By the definition of $T_0$ we get assertion 1, which on the other hand implies that in a neighborhood of each $t\in(0,T_0)$ all the derivatives of $x^n$ are uniformly bounded, which by Arzela--Ascoli theorem implies that $x$ is smooth in $(0,T_0)$. With this, to prove that $x$ solves C-S system with weight $\psi$, it suffices to take a $(C^2([t-\epsilon,t+\epsilon]))^{Nd}$ limit of systems associated with weights $\psi_n$, with $[t-\epsilon,t+\epsilon]\subset(0,T_0)$.
\end{proof}
\noindent
Our next step is to show that the function $x$ actually satisfies our problem in a weak sense in every interval $[0,t]\subset[0,T_0)$ (though to prove that it satisfies Definition \ref{sol} we still nead continuity of $v$ at $T_0$).

\begin{prop}\label{spelnia}
For all $t\in[0,T_0]$ the function $x$ is a weak $(W^{2,1}([0,t]))^{Nd}$ solution of (\ref{cs}).
\end{prop}
\begin{proof}
From Proposition \ref{odc} and Corollary \ref{coroodc} we know that $x\in (C^1([0,T_0)))^{Nd}$ and that $T_0$ is the time of the first collision of the particles. It suffices to show that $x$ satisfies (\ref{cs}) weakly in intervals $[0,t]$ for $t\nearrow T_0$. Since $x^n$ satisfy $(\ref{cs})_1$ and $x^n\to x$ in $(C^1([0,t]))^{Nd}$, then $x$ satisfies $(\ref{cs})_1$ with $v=\lim_{n\to\infty}v^n$. Now for $\phi\in (C_c^\infty([0,t]))^d$, we have
\begin{eqnarray*}
\int_0^{t}v_i^n\dot{\phi} ds =-\int_0^{t}\dot{v_i^n}\phi ds=-\int_0^{t}\frac{1}{N}\sum_{k=1}^N(v_k^n-v_i^n)\psi_n(|x_i^n-x_k^n|)\phi ds
\end{eqnarray*}
and the left-hand side converges to $\int_0^{t} v\dot{\phi} ds$. Thus it remains to show that the right-hand side converges to $[-\int_0^{t}\dot{v}\phi ds]$, where
\begin{eqnarray}\label{podzial}
\dot{v}:=\frac{1}{N}\sum_{k=1}^N(v_k-v_i)\psi(|x_i-x_k|).
\end{eqnarray}
To this end we require for example that $\dot{v}^n\rightharpoonup\dot{v}$ in $(L^1([0,t]))^{Nd}$, which follows from Lemma \ref{weakstrong} applied to functions $f_n=v^n_k-v^n_i$, $f=v_k-v_i$, $g_n=\psi(|x^n_i-x^n_k|)$, $g=\psi(|x_i-x_k|)$.
\end{proof}
\begin{rem}\rm
After arriving at (\ref{podzial}) we may apply a stronger argument that $\dot{v^n}\to \dot{v}$ in $(L^1([0,t]))^{Nd}$. Clearly $x^n\to x$ and $v^n\to v$ a.e. and thus also $\dot{v^n}\to\dot{v}$ a.e.. Moreover $\dot{v^n}\in (L^1([0,t]))^{Nd}$ for all $n$. Therefore if we show that functions $v^n$ are uniformly integrable, then by Lemma \ref{vit} the proof will be finished. Since $\dot{v^n}$ are uniformly bounded in $[\delta,t]$ for all $\delta>0$ (as there are no collisions in $(0,T_0)$\footnote{perhaps a better argument is that it follows from Lemma \ref{equi} and property 4 from Proposition \ref{properties}}), then the set $\{0\leq s\leq t: |\dot{v^n}|>c\}$ is included in some interval $[0,s_c]$ with $s_c\to 0$ as $c\to\infty$. Therefore by Lemma \ref{equi} we get
\begin{eqnarray*}
\int_0^{t}|\dot{v^n}|\chi_{\{s:|\dot{v^n}|>c\}} ds\leq \int_0^{s_c}|\dot{v^n}|ds\leq L|s_c|^{1-\alpha}\to 0
\end{eqnarray*}
as $c\to\infty$, which proves that $\dot{v^n}$ are uniformly integrable.
\end{rem}
\noindent
As our last effort in this section let us make an obvious remark involving properties stated in Proposition \ref{properties}.

\begin{coro}\label{properties2}
Properties 1,2,6,7 from Proposition \ref{properties} remain true also for the solution $x$ on $[0,T_0)$. Moreover the following version of properties 3 and 4 holds:
\begin{description}
\item[{\it $(4^{'})$}]
For all initial data $x(0)$ and $v(0)$ and all $i=1,...,N$, we have
\begin{eqnarray*}
\|v_i\|_{L^\infty([0,T_0))}\leq M,
\end{eqnarray*}
where $M$ is the constant from property 4 from Proposition \ref{properties}. 
\end{description}
\end{coro}
\begin{proof}
Properties 1,2,$4^{'}$ follow by similar argumentation as in the proof of Proposition \ref{properties}. Property 6 follows by definition of our system (namely by substituing equation (\ref{cs2}) with $(\ref{cs})_2$) and property 7 follows by calculating the derivative of
\begin{eqnarray*}
r(t)=\sum_{i,j}(v_i-v_j)^2.
\end{eqnarray*}
\end{proof}
\begin{rem}\rm
Property 5 from Proposition \ref{properties} clearly does not hold on $[0,T_0)$ even though it holds on $[0,t]$ for $t\nearrow T_0$ (to see this, we simply substitute in property 5, $n$ with $\lfloor\delta_t\rfloor+2$, where $\delta_t$ is defined in part {\bf (A)} of the proof of Proposition \ref{odc}).
\end{rem}

\subsection{Clustering at the time of collision}\label{clus}
In the previous section we established existence of solutions on the interval $[0,T_0)$, where $T_0$ is time of the first collision of some pair of particles. The solution $x$ belongs to $(W^{2,1}([0,t]))^{Nd}\cap (C^{1}([0,T_0)))^{Nd}\cap (C([0,T_0]))^{Nd}$ for all $0<t<T_0$ and satisfies (\ref{cs}) in a classical sense in $(0,T_0)$ and weakly in $(W^{2,1}([0,t]))^{Nd}$. Therefore we know that $v$ is a Lipschitz continuous function in each interval $[0,t]\subset[0,T_0]$, however we do not know anything about it's behaviour in a neighborhood of $T_0$ -- with our current knowledge the limit of $v(t)$ as $t\to T_0$ may even not exist. In this section we provide a proof of continuity of $v$ on whole interval $[0,T_0]$.
\begin{defi}
For each $i,j=1,...,N$ we define a relation $i\dot{\sim} j$ if and only if $j\notin B_i(0)$ or for all $t<T_0$, we have
\begin{eqnarray*}
\int_t^{T_0}\psi(|x_i-x_j|)ds=\infty.
\end{eqnarray*}
\end{defi}
\noindent
This relation is clearly symetric and reflexive but not necessarily transitive. This leads us to another definition.
\begin{defi}
For each $i,j=1,...,N$ we define a relation $\sim$ with the following two statements:
\begin{enumerate}
\item If $i\dot{\sim} j$, then $i\sim j$.
\item For $i\dot{\nsim} j$, we have $i\sim j$ if and only if there exists $k$, such that $i\sim k$ and $k\sim j$. 

\end{enumerate}
\end{defi}
\begin{rem}\rm\label{inti}
Relation $\sim$ is an equivalence relation. Since $\dot{\sim}$ is symetric and reflexive then so is $\sim$. Transitivity of $\sim$ follows directly from the definition. Equivalence classes $[i]$ of $\sim$ provide us with a partition of the set of indexes $\{1,...,N\}$ with the following property: given $i,j=1,...,N$ if $j\notin [i]$, then $\psi(|x_i-x_j|)$ is integrable in every interval $[t,T_0]$.  
\end{rem}
\noindent
Now let us for each $i=1,...,N$ define $w_i=w^{t_0}_i$ by the system of ODE's
\begin{eqnarray*}
\dot{w_i}=\frac{1}{N}\sum_{k\in [i]}(w_k-w_i)\psi(|x_i-x_k|)
\end{eqnarray*}
in $[t_0,T_0)$ with the initial data $w_i(t_0)=v_i(t_0)$ for all $i=1,...,N$. All structure based properties $1,2$ and $4^{'}$ from Corollary \ref{properties2} hold also for the functions $w_i$ as in their proof we never make use of the fact that $\dot{x}=v$. We introduce the functions $w_i$ as a tool to study the evolution of $v$ in a neighborhood of $T_0$. First we ensure that $w_i$ and $v_i$ are somehow close to each other and behave in a similar way.
\begin{prop}\label{podob}
For $t\in[t_0,T_0)$, we have
\begin{eqnarray*}
|v_i(t)-w_i(t)|\leq \omega(T_0-t_0),
\end{eqnarray*}
for some nonnegative continuous function $\omega$ with $\omega(0)=0$.
\end{prop}
\begin{proof}
Let $r(t)=\sum_{i\in[i]}(v_i(t)-w_i(t))^2$. We have
{\small
\begin{eqnarray*}
r^{'}=\frac{2}{N}\sum_{i,j\in[i]}\left(v_i-w_i\right)\left((v_j-v_i)-(w_j-w_i)\right)\psi(|x_i-x_j|) + \frac{2}{N}\sum_{i,j\notin[i]}(v_i-w_i)(v_j-v_i)\psi(|x_i-x_j|)=: I+II.
\end{eqnarray*}
}
By the usual symetry argument
\begin{eqnarray*}
I=\frac{2}{N}\sum_{i,j\in[i]}\left((v_i-w_i)(v_j-w_j)-(v_i-w_i)^2\right)\psi(|x_i-x_j|)=\\
-\frac{1}{N}\sum_{i,j\in[i]}\left((v_i-w_i)-(v_j-w_j)\right)^2\psi(|x_i-x_j|)\leq 0.
\end{eqnarray*}
On the other hand $II$ is integrable by Remark \ref{inti}. Therefore, since $r(t_0)=0$, for $t\in[t_0,T_0)$, we have
\begin{eqnarray*}
r(t)\leq \int_{t_0}^{T_0} |II|ds=:\omega^2(T_0-t_0),
\end{eqnarray*}
 where $\omega$ is a nonnegative continuous function with $\omega(0)=0$.
\end{proof}
\noindent
Our next goal is to prove that if $i\sim j$ then $|w_i(t)-w_j(t)|\to 0$ as $t\to T_0$. However before we begin let us make another purely technical assumption that
\begin{eqnarray}\label{asum}
\sum_{i\in[i]}w_i=0.
\end{eqnarray}
This does not make our reasoning any less general since by property 2 from Corollary \ref{properties2} this sum is constant in time -- thus we may as well assume that it equals $0$. Thus our goal can be rewritten in a equivalent form: prove that
\begin{eqnarray}\label{equiv}
\lim_{t\to T_0}w_i(t)=0 \ \ \  {\rm for\ all}\ \ \ i\in[i].
\end{eqnarray}
The first step of the proof is to show the following slightly weaker assertion.

\begin{lem}
If $i\dot{\sim} j$, then there exists a sequence $s_n\to T_0$, such that $|w_i(s_n)-w_j(s_n)|\to 0$.
\end{lem}

\begin{proof}
The proof follows by contradiction. Let us assume that $i\dot{\sim} j$ and there is no such sequence $s_n$ i.e. there exists $\delta>0$, such that $|w_i(s)-w_j(s)|>\delta$ for $s\in[t_0,T_0)$. Since $i\dot{\sim} j$ both $i$ and $j$ belong to $[i]$ and thus for all $s\in [t_0,T_0)$ and for $r(s):=\sum_{k,l\in[i]}(w_k(s)-w_l(s))^2$, we have
\begin{eqnarray*}
r^{'}=\frac{2}{N}\sum_{k,l,m\in[i]}(w_k-w_l) \left((w_m-w_k)\psi(|x_k-x_m|)-(w_m-w_l)\psi(|x_l-x_m|)\right).
\end{eqnarray*}
By a similar to the proof of property 3 form Proposition \ref{properties} application of the symetry we conclude that
\begin{eqnarray*}
r^{'}=-2\sum_{k,l\in[i]}(v_k-v_l)^2\psi(|x_k-x_l|).
\end{eqnarray*}
Now since $|w_i-w_j|>\delta$ and by property $4^{'}$ from Corollary \ref{properties2} also $\delta^2<r(s)\leq NM^2$ and we have
\begin{eqnarray*}
(\ln r)^{'}\leq -2\frac{(v_i-v_j)^2}{r}\psi(|x_i-x_j|)\leq -\frac{2\delta^2}{NM^2}\psi(|x_i-x_j|)
\end{eqnarray*}
and consequently
\begin{eqnarray*}
\delta^2<r\leq e^{-\frac{2\delta^2}{NM^2}\int_{t_0}^s\psi(|x_i-x_j|)d\theta} r(t_0),
\end{eqnarray*}
which is impossible since $\int_{t_0}^s\psi(|x_i-x_j|)\to\infty$ as $s\to T_0$. Therefore no such $\delta$ exists and the proof is complete.
\end{proof}
\noindent
Our next step is a technical lemma which is vaguely based on the fact that velocities of the particles only "pull" each other but never push away (which for example means that $w_i$ which is the furthest from $0$ may not go any further away from $0$ because there is no other velocity to pull it there).

\begin{lem}\label{grav}
For each $k=1,...,d$ we denote $w^k_i$ -- the $k$-th coordinate of $w_i$ and assume that up to permutations $w_1^k(t)\leq...\leq w_N^k(t)$. Then the sums
\begin{eqnarray*}
\sum_{i=1}^lw_i^k(t),\ \ \ {\rm and}\ \ \ \sum_{i=l}^Nw_i^k(t),\ \ \ l=1,...,N
\end{eqnarray*} 
are respectively nondecreasing and nonincreasing. 
\end{lem}

\begin{proof}
We prove the assertion only for the first sum as the other differs only by sign. For all $l=1,...,N$, we have
\begin{eqnarray*}
\left(\sum_{i=1}^lw^k_i\right)^{'}=\sum_{i,j=1}^l(w^k_j-w^k_i)\psi(|x_i-x_j|)+ \sum_{i=1}^l\sum_{j=l+1}^N(w^k_j-w^k_i)\psi(|x_i-x_j|)=: I+II.
\end{eqnarray*}
By symetry $I=0$. On the other hand for $j>l$ as long as $w^k_j-w^k_i>0$, we have $II\geq 0$ and the sum $\sum_{i=1}^lw^j_i$ is nondecreasing.
\end{proof}

\noindent
Now we may proceed with our goal which is the following proposition.

\begin{prop}\label{zbieg}
If $i\sim j$ then 
\begin{eqnarray}\label{tion}
\lim_{t\to T_0}|w_i(t)-w_j(t)|=0.
\end{eqnarray}

\end{prop}

\begin{proof}
It suffices to show that the assertion holds if we substitute $w_i$ with $w^k_i$ -- it's $k$-th coordinate, thus let us assume for simplicity of notation that $w_i=w^k_i$. Therefore $w_i$ are real functions and their sum equals to $0$ by (\ref{asum}). The proof follows by 3 steps.\\
{\sc Step 1.} For $t\in[t_0,T_0)$, let
\begin{eqnarray*}
\mathcal{R}(t):=\max_{j\in[i]}w_j(t).
\end{eqnarray*}
\noindent
First we prove that if at some point $t\in[t_0,T_0)$ we have
\begin{eqnarray}\label{d}
w_i(t)=\mathcal{R}(t) - \delta,
\end{eqnarray}
then
\begin{eqnarray}\label{obs}
\sup_{s\in[t,T_0)}w_i\leq \mathcal{R}(t) - \frac{\delta}{N!}.
\end{eqnarray}
This means that if some velocity $w_i$ is far away from the highest velocity at the time $t$, then the highest possible value for $w_i$ is significantly smaller than the highest velocity at the time $t$. The proof follows by induction with respect to the number of velocities $w_j$ that are bigger than $w_i$ at the time $t$. For $n=1$ we are in a situation when there is only one $w_j$, such that $\mathcal{R}(t)=w_j(t)>w_i(t)$ and (\ref{d}) implies that $\mathcal{R}(t)-w_i(t)=\delta$. Now let
\begin{eqnarray*}
p(s):=\max\{w_k(s): w_k(s)<\mathcal{R}(s)\},\ \ \  {\rm for}\ \ \ s\in[t,T_0)
\end{eqnarray*}
Clearly $p(t)=w_i(t)$ but it is possible that some other velocity may become bigger than $w_i$ at some point in time and this is the only reason to introduce the function $p$, which will serve us by pointing the right-hand edge of the set of velocities smaller than $\mathcal{R}$. Clearly $w_i\leq p\leq \mathcal{R}$ in $[t,T_0)$. Moreover Lemma \ref{grav} implies that the sum $p+\mathcal{R}$ is nonincreasing. Therefore
\begin{eqnarray*}
\mathcal{R}(t)\geq \mathcal{R}(s)+p(s)-p(t)\geq 2w_i(s) - w_i(t) =  2w_i(s) - \mathcal{R}(t)+\delta,
\end{eqnarray*}
which implies that
\begin{eqnarray*}
\sup_{s\in[t,T_0)}w_i\leq \mathcal{R}(t)-\frac{\delta}{2}.
\end{eqnarray*}
Now let us assume that condition (\ref{d}) implies that
\begin{eqnarray}\label{indasum}
\sup_{s\in[t,T_0)}w_i\leq \mathcal{R}(t)-\frac{\delta}{(n+1)!}
\end{eqnarray}
in case when at the time $t$ there are exactly $n$ velocities bigger than $w_i$. We will prove that this implies that if (\ref{d}) holds, then
\begin{eqnarray}\label{tez}
\sup_{s\in[t,T_0)}w_i\leq \mathcal{R}(t)-\frac{\delta}{(n+2)!}
\end{eqnarray}
if only there are exactly  $n+1$ velocities bigger than $w_i$ at the time $t$. In such case we define
\begin{eqnarray*}
p(s):=\max\{w_k(s): k\notin \mathcal{G}\},\ \ \ {\rm for}\ \ \ s\in[t,T_0),
\end{eqnarray*}
where $\mathcal{G}$ is the set of indexes of the $n+1$ biggest velocities at the time $t$. Here again the sole purpose of the function $p$ is to point the biggest velocity that was initialy smaller the the biggest $n+1$ velocities. Denoting $\sum_{k\in \mathcal{G}}w_k(s)=:S(s)$, by Lemma \ref{grav}, the function $S+p$ is nonincreasing as long as
\begin{eqnarray}\label{p}
p(s)< \min_{k\in \mathcal{G}}w_k(s),
\end{eqnarray}
thus
\begin{eqnarray*}
(n+2)p(s)< S(s)+p(s)\leq S(t)+p(t)= S(t)+\mathcal{R}(t) - \delta\leq (n+2)\mathcal{R}(t) - \delta,\\
p(s)<\mathcal{R}(t) - \frac{\delta}{n+2}
\end{eqnarray*}
as long as (\ref{p}) holds. So if at some time $s_0$, we have $p(s_0)=\mathcal{R}(t) - \frac{d}{n+2}$ then also $p(s_0)\geq \min_{k\in \mathcal{G}}w_k(s_0)$. At that point there are at most $n$ velocities bigger than $p$ and the distance between $p(s_0)$ and $\mathcal{R}(s_0)$ is no less than $\delta^{'}:=\frac{\delta}{n+2}$. Therefore by (\ref{indasum}), we have
\begin{eqnarray*}
p(s)\leq \mathcal{R}(s_0)-\frac{\delta^{'}}{(n+1)!}\leq \mathcal{R}(t)-\frac{\delta}{(n+2)!} \ \ \ {\rm for} \ \ \ s\in[s_0,T_0).
\end{eqnarray*}
This proves (\ref{tez}). Noticing that $n\leq N-1$ we get (\ref{obs}) and finish step 1.\\ 
{\sc Step 2.} Our next step is the following simple observation with the proof left for the reader.
\begin{lem}\label{step2}
If (\ref{tion}) does not hold, then there exists $\epsilon>0$ and a sequence $s_n\to T_0$, such that for all $i\in [i]$ there exists $j\in[i]$, such that
\begin{eqnarray*}
|w_i(s_{n_k})-w_j(s_{n_k})|\to 0\ \ \ {\rm and}\ \ \ |w_i(s_{n_l})-w_j(s_{n_l})|>\epsilon
\end{eqnarray*}
for some subsequences $\{s_{n_k}\},\{s_{n_l}\}\subset \{s_n\}$.
\end{lem}

{\sc Step 3.} Now we finish the prove. Let us fix $t\in[t_0,T_0)$ and assume that $w_i$ is one of the biggest velocities at the time $t$ i.e. $\mathcal{R}(t)=w_i(t)$. Lemma \ref{step2} ensures existence of $j$, such that 
\begin{eqnarray}\label{seconv}
|w_i-w_j|\to 0
\end{eqnarray}
on one subsequence converging to $T_0$ and
\begin{eqnarray}\label{dieps}
|w_i-w_j|>\epsilon
\end{eqnarray}
on some other subsequence converging to $T_0$ for $\epsilon$ independent of $i$ and $j$. Thus (\ref{dieps}) implies that at some time $s\in[t,T_0)$ either $w_i$ or $w_j$ (say $w_j$) is farther from $\mathcal{R}(t)$ than $\epsilon$. Then step 1 implies that
\begin{eqnarray*}
\sup_{\theta\in[s,T_0)}w_j\leq \mathcal{R}(t)-\frac{\epsilon}{N!}.
\end{eqnarray*}
Moreover (\ref{seconv}) implies that at some other time $r\in[s,T_0)$, we have
\begin{eqnarray*}
w_i(r)\leq \mathcal{R}(t)-\frac{\epsilon}{2N!}
\end{eqnarray*}
and after that point (again by step 1)
\begin{eqnarray*}
\sup_{\theta\in[r.T_0)}w_i\leq \mathcal{R}(t)-\frac{\epsilon}{(2N!)^2}.
\end{eqnarray*}
This procedure can be performed with any velocity $w_i$ that at some time equals to $\mathcal{R}$ as many times as we want. Therefore we may make sure that $\mathcal{R}(t)$ is arbitrarily small at some time $t<T_0$. The same can be done with $\mathcal{L}(t):=\min_{j\in[i]}w_j(t)$ to conclude that diamaterer of velocities converges to $0$ as $t\to T_0$ and this contradicts (\ref{dieps}) and by Lemma \ref{step2} implies that assertion (\ref{tion}) is true. This finishes the proof.
\end{proof}

\begin{rem}\rm\label{drop}
In Proposition \ref{zbieg} we proved that for all $i\in[i]$ we have $w_i\to 0 $. However this was under our assumption (\ref{asum}). Now it is time to drop this assumption and conclude that in general there exists a constant $\bar{v}$, such that for all $i\in[i]$, we have $w_i\to \bar{v} $.
\end{rem}
\noindent
Our last goal in this subsection is to clarify what does Proposition \ref{zbieg} imply to the motion of $v$.

\begin{coro}\label{no}
There exists a constant $\bar{v}$, such that for all $i\in[i]$, we have
\begin{eqnarray*}
\lim_{t\to T_0^-}v_i(t)=\bar{v}.
\end{eqnarray*}
\end{coro}

\begin{proof}
Given $\epsilon>0$ we need to ensure existence of $s_0<T_0$, such that for all $s_0<s<T_0$, we have
\begin{eqnarray*}
|v_i(s)-\bar{v}|<\epsilon 
\end{eqnarray*}
Let $t_0$ be such that $\omega(T_0-t_0)<\frac{\epsilon}{2}$, where $t_0$ and $\omega$ are as in Proposition \ref{podob}. By Remark \ref{drop} there exists $s_0\in[t_0,T_0)$, such that for all $s\in[s_0,T_0)$, we have
\begin{eqnarray*}
|v_i(s)-\bar{v}|\leq |v_i(s)-w^{t_0}_i(s)|+|w^{t_0}_i(s)-\bar{v}|<\omega(T_0-t_0)+\frac{\epsilon}{2}<\epsilon.
\end{eqnarray*}
\end{proof}
\noindent
This finally proves that the function $v$ has a limit at $T_0^-$ and we may extend it continuously to $[0,T_0]$. For the sake of clarity of argumentation in the next section, it is useful to summarise what we proved in this section.
\begin{rem}\rm\label{sum}
We actually proved that
\begin{enumerate}
\item If
\begin{eqnarray*}
\int_t^{T_0}\psi(|x_i-x_j|)ds = \infty
\end{eqnarray*}
for all $t<T_0$, then $v_i(t)-v_j(t)\to 0$ as $t\to T_0$
\item If on the other hand $v_i$ is separated from $v_j$ in a left sided neighborhood of $T_0$ then we have
\begin{eqnarray}\label{omega}
\int_t^{T_0}\psi(|x_i-x_j|)ds < \infty
\end{eqnarray}
\item Condition (\ref{omega}) holding for all $i$ and $j\in B_i(0)$ implies that all functions $v_i$ are in fact weak solutions of $(\ref{cs})_2$ in $W^{1,1}([0,T_0])$ and in particular admit a modulus of continuity and are uniformly continuous at $T_0^-$. Therefore again there exists a limit of $v_i$ at $T_0$ but it does not necessarily equal to a limit of $v_j$ for any $j\in B_i(0)$. 

\end{enumerate}
\end{rem}

\subsection{Global existence}
In this section we combain our efforts from sections \ref{exi} and \ref{clus} to obtain global existence in the sense of Definition \ref{sol}. Propositions \ref{odc} and \ref{spelnia} ensure existence of weak solutions in $[0,T_0)$ with a continuous velocity $v$. Corollary \ref{no} implies that there exists a left sided limit of $v$ at $T_0$. Thus it may be extended continuously to $[0,T_0]$. Therefore for arbitrary initial data there exists a unique solution $x\in C^{1}([0,T_0])^{dN}$ satisfying Definition \ref{sol}. Now assuming that $T_0$ is a new initial point with initial data equal to $x(T_0)$ and $v(T_0)$ and aplying Proposition \ref{odc} we conclude that the solution exists on $[0,T_n]$, where $T_n$ is $n+1$-th time at which some particles collide. What is not clear however is whether for arbitrary $T$ we may find $T_n\geq T$. 

\begin{prop}\label{exif}
For all $T>0$ there exists a $(C^1([0,T]))^{Nd}$ solution of (\ref{cs1}) with arbitrary initial data. This solution is in the sense of Definition \ref{sol}.
\end{prop}
\begin{proof}
It suffices to show that we may extend our solution up to an arbitrary $T>0$. Let $T_n$ be a sequence of the points of collision and assume by contradiction that $\sum_n(T_n-T_{n-1})<\infty$. Then $T_n-T_{n-1}\to 0$ and $T_n\to \tilde{T}$ for some $\tilde{T}>0$. We will prove that $\tilde{T}$ is a point of sticking for some particles $x_i$ and $x_j$. Clearly there exist $i$ and $j\in B_i(0)$ and a subsequence $T_{n_k}$, such that $x_i(T_{n_k})-x_j(T_{n_k})=0$, which by Lipschitz continuity of $x$ (property $4^{'}$ from Corollary \ref{properties2}) implies that $x_i(t)-x_j(t)\to 0$ as $t\to\tilde{T}$ and $\tilde{T}$ is a point of collision of $x_i$ and $x_j$. Now it remains to show that $v_i(t)-v_j(t)\to 0$ as $t\to\tilde{T}$. If
\begin{eqnarray}
\int_t^{\tilde{T}}\psi(|x_i-x_j|)ds = \infty
\end{eqnarray}
for all $t<\tilde{T}$, then by Remark \ref{sum},1 we are done. On the other hand if $\psi(|x_i-x_j|)$ is integrable in a left sided neighborhood of $\tilde{T}$ then by Remark \ref{sum},3, velocity $v$ is uniformly continuous at $\tilde{T}^-$ and in particular has a limit at $\tilde{T}$. Therefore there exists a limit of $v_i-v_j$ at $\tilde{T}$. If this limit equals to $0$ then, again, we are done. If on the other hand it equals to some $\xi\neq 0$, then in a neigborhood of $\tilde{T}$ we have $v_i-v_j\in B(\xi,\epsilon)$, where $B(\xi,\epsilon)$ is a ball centered at $\xi$ with an arbitrary small radius $\epsilon$. This implies a clearly false statement that 
\begin{eqnarray*}
0=x_i(T_{n_{k+1}})-x_j(T_{n_{k+1}})\in (T_{n_{k+1}}-T_{n_k})B(\xi,\epsilon)
\end{eqnarray*}
with $T_{n_k}$ and $T_{n_{k+1}}$ sufficiently close to $\tilde{T}$. This contradicts the assumption that $\xi\neq 0$. Altogether we proved that if $\tilde{T}$ is a density point for $T_n$ then it is a point of sticking of the particles. Then we may further extend our solution beyond $\tilde{T}$. Finally, since there can be at most $N-1$ times of sticking, then for all $T>0$ either we can find $T_n$ such that $T_n>T$ or all the particles stick together before time $T$ and travel with constant velocity for as long as we want them to.
\end{proof}

\section{On the case of two particles -- flocking in a finite time}
In this section our goal is to discuss the possibility of a finite in time alignment in case of two particles ($N=2$). First let us recall that asymptotic flocking was studied before in most papers mantioned in the introduction, see e.g. \cite{haliu} and we refer to those papers to see general definitions and results. Here, we consider the most strict form of flocking, which is sticking of the trajectories of the particles in a finite time. By property 2 from Corollary \ref{properties2} the avarage velocity of the particles is constant, which means that
\begin{eqnarray*}
v_1\equiv -v_2+\bar{v}
\end{eqnarray*}  for some constant $\bar{v}$. Without a loss of generality we may assume that $\bar{v}=0$. The above observation implies that
\begin{eqnarray*}
x_1(t)=-x_2(t)+t\bar{v}+(x_1(0)+x_2(0))
\end{eqnarray*}
 and assuming without a loss of generality that also $x_1(0)=-x_2(0)$, we get $x_1\equiv-x_2$. Thus both, avarage velocity and the center of mass of the particles are equal to $0$. Therefore the particles move parallely to each other, either on two separate parallel lines or on the same line. In the former case, the distance between particles is always no less than the distance of respective lines, thus there is no possibility of a finitie in time (or asymptotic for that matter) alignment. In the latter case the distance between particles can by arbitrarily small, thus hypothetically a finite in time alignment may occur. In order to simplify our calculations, since particles move on the same line, then by a simple change of variables we may assume that $d=1$. Altogether we have two particles $x_1$ and $x_2$, with $x_1\equiv-x_2$ and $v_1\equiv -v_2$. Therefore they are unequivocally defined by the function
\begin{eqnarray*}
\phi(t):=x_2(t)-x_1(t).
\end{eqnarray*}
Then the C-S model (\ref{cs}) (or (\ref{cs1}) and (\ref{cs2}), since in this case they are the same) can be rewritten equivalently as
\begin{eqnarray}\label{raz}
\ddot{\phi}(t)=-2\dot{\phi}(t)\psi(|\phi(t)|),
\end{eqnarray}
with $\phi(0)=x_2(0)-x_1(0)\geq 0$ and $\dot{\phi}(0)=v_2(0)-v_1(0)\in\r$. Moreover Lemma \ref{grav} implies that if at some time $t$ we have $\dot{\phi}(t)=0$ then it will be constantly equal to $0$ from that point in time. This implies that $\dot{\phi}$ may not change sign and this farther implies that there may be at most one collision of the particles. 
Finally let us notice that by Proposition \ref{exif} there exists a solution to (\ref{raz}) with arbitrary initial data and we can easly prove that if $\phi(0)>0$, then this solutions is unique. Now we are ready to state our main result of this section.

\begin{prop}\label{stick}
Let $\phi$ be a solution of (\ref{raz}) with $\phi(0)>0$. Then the following are equivalent:
\begin{enumerate}
\item There exists a time $t_0<\infty$ such that $\phi(t_0)=\dot{\phi}(t_0)=0$.
\item Initial data satisfy:
\begin{eqnarray}\label{piec}
\dot{\phi}(0)=-2\Psi(\phi(0)),
\end{eqnarray}
where $\Psi(s):=\frac{1}{1-\alpha}s^{1-\alpha}$ is a primitive of $\psi$.

\end{enumerate}
\end{prop}
\begin{proof}
Since there is at most one collision of the particles and we know that they stick together, thus $\phi$ and $\dot{\phi}$ have constant signs. Therefore since $\phi(0)>0$ then also $|\phi|=\phi$ and by simple integration of (\ref{raz}) we conclude that the function $\phi$ satisfies:
\begin{eqnarray}\label{trzy}
\dot{\phi}(t)=-2\Psi(\phi(t))+2\Psi(\phi(0))+\dot{\phi}(0)
\end{eqnarray}
\noindent
and
\begin{eqnarray}\label{cztery}
\phi(t)=-2\int_0^t\Psi(\phi(s))ds + t(2\Psi(\phi(0))+\dot{\phi}(0))+ \phi(0).
\end{eqnarray}

Substituing $t$ with $t_0$ in (\ref{trzy}) we obtain
\begin{eqnarray*}
0=2\Psi(\phi(0))+\dot{\phi}(0),
\end{eqnarray*}
which is exatly condition (\ref{piec}). Now let as assume that (\ref{piec}) is satisfied. We are going to prove existence of $t_0$. First note that in our case (\ref{trzy}) and (\ref{cztery}) are satisfied on the set $\{t:\phi(t)\geq 0\}$ and they have the following form:
\begin{eqnarray}
\dot{\phi}(t)&=&-2\Psi(\phi(t)),\label{szesc}\\
\phi(t)&=&-2\int_0^t\Psi(\phi(s))ds + \phi(0).\nonumber
\end{eqnarray}
From (\ref{szesc}) and by the definision of $\psi$ we obtain
\begin{eqnarray*}
\dot{\phi}(t)=-\frac{2}{1-\alpha}\phi(t)\psi(\phi(t))
\end{eqnarray*}
and
\begin{eqnarray*}
\phi(t)=e^{-\frac{2}{1-\alpha}\int_0^t\psi(\phi(s))ds}\phi(0).
\end{eqnarray*}
Thus, since $\max_{t\in[0,t_0]}\phi(t)=\phi(0)$, we have
\begin{eqnarray*}
\phi(t)\leq e^{-\frac{2}{1-\alpha}t\psi(\phi(0))}\phi(0),
\end{eqnarray*}
which can become arbitrarily small in a finite time. Now for $n=2,3,...$ let
\begin{eqnarray*}
t_n:=\inf\{t>t_{n-1}:\phi(t)\leq 2^{-n}\},
\end{eqnarray*}
with $t_1:=0$. We have
\begin{eqnarray*}
\phi(t_n)&=&e^{-\frac{2}{1-\alpha}\int_{t_{n-1}}^{t_n}\psi(\phi(s))ds}\phi(t_{n-1}),\\
2^{-1}&=&e^{-\frac{2}{1-\alpha}\int_{t_{n-1}}^{t_n}\psi(\phi(s))ds},\\
\ln 2&=&\frac{2}{1-\alpha}\int_{t_{n-1}}^{t_n}\psi(\phi(s))ds\geq \frac{2}{1-\alpha}(t_n-t_{n-1})2^{\alpha(n-1)}.
\end{eqnarray*}
Therefore
\begin{eqnarray*}
(t_n-t_{n-1})\leq \frac{(1-\alpha)\ln 2}{2}2^{\alpha(1-n)}
\end{eqnarray*}
and $t_n$ is a partial sum of a convergent series. Thus $t_n$ converges to a finite limit $t_0$ such that $\phi(t_0)=\dot{\phi}(t_0)=0$.
\end{proof}

\begin{rem}\rm
Finally let us mention that a finite in time allignemnt may not happen in case of $\psi_{cs}$ defined by (\ref{cucu}) since (\ref{raz}) implies that
\begin{eqnarray*}
|\dot{\phi}(t)|=e^{-2\int_0^t\psi_{cs}(|\phi(t)|)ds}\dot{\phi}(0)\geq e^{-2t\|\psi_{cs}\|_\infty}|\dot{\phi}(0)|>0,
\end{eqnarray*}
as long as $\dot{\phi}(0)\neq 0$. Moreover we may just as easly prove that with unintegrable singular weight $\psi$, e.g. when $\psi(s)=s^{-\alpha}$ for $\alpha>1$ in one dimensional setting not only particles cannot stick but they cannot even collide.
\end{rem}

\noindent
Conditions described in Proposition \ref{stick} refer to the function $\phi$ and in a simplified case of one dimension. However they can be modified to cover more general cases and refer directly to $x_1$ and $x_2$.\\
\noindent
{\bf Acknowledgement.} This work was partialy supported by MNiSW grant no. IdP2011 000661.

\end{document}